\documentclass[reqno,11pt]{amsart}
\usepackage[english]{babel}
\usepackage{amsmath,amssymb,amsbsy,amsfonts,amsthm,latexsym,amsopn,amstext,amsxtra,euscript,amscd}

\usepackage{array}
\usepackage{ifthen}
\usepackage{url}
\usepackage{graphicx}

\usepackage{hyperref}

\newtheorem{theorem}{Theorem}
\newtheorem{proposition}[theorem]{Proposition}
\newtheorem{conjecture}{Conjecture}

\theoremstyle{definition}

\renewcommand{\leq}{\leqslant}
\renewcommand{\geq}{\geqslant}
\renewcommand{\le}{\leqslant}
\renewcommand{\ge}{\geqslant}

\begin{document}

\title{On the number of divisors of Mersenne numbers}

\author[V. Kova\v{c}]{Vjekoslav Kova\v{c}}
\address{V. K., Department of Mathematics, Faculty of Science, University of Zagreb, Bijeni\v{c}ka cesta 30, 10000 Zagreb, Croatia}
\email{vjekovac@math.hr}

\author[F. Luca]{Florian Luca}
\address{F. L., Stellenbosch University, Mathematics, Merriman Street 7600 Stellenbosch, South Africa
and Max Planck Institute for Software Systems, Saarbr\"{u}cken, Germany}
\email{fluca@sun.ac.za}


\pagenumbering{arabic}

\begin{abstract}
Denote $f(n):=\sum_{1\le k\le n} \tau(2^k-1)$, where $\tau$ is the number of divisors function. Motivated by a question of Paul Erd\H{o}s, we show that the sequence of ratios $f(2n)/f(n)$ is unbounded. We also present conditional results on the divergence of this sequence to infinity. Finally, we give experimental evidence on both the conjecture $f(2n)/f(n)\to\infty$ and our sufficient conditions for it to hold.
\end{abstract}

\keywords{divisor function, Mersenne number, asymptotics, highly-composite number, cyclotomic polynomial}

\maketitle

\section{Introduction}

\subsection{Problem background}
The following problem is due to Paul Erd\H{o}s, who posed it in the proceedings \cite{893} of the conference \emph{Number Theory}, held in Eger, Hungary in 1996. Let 
$$
f(n):=\sum_{1\le k\le n} \tau(2^k-1)
$$
be the total number of divisors of the first $n$ \emph{Mersenne numbers} $2^k-1$.
After expressing the feeling that $f(n)$ increases too fast to allow an asymptotic formula as $n\to\infty$, Erd\H{o}s asked:
\begin{quote}
\emph{Is it true that
$$
f(2n)/f(n)
$$
tends to a limit?} \cite[p.\,180]{893}
\end{quote}
Recently, the question was also posed as Problem \#893 on Thomas Bloom's website \emph{Erd\H{o}s problems} \cite{EP}.

Erd\H{o}s did not write down his motivation to formulate this problem, but he was generally interested in the behaviour of the ``doubling ratio'' $f(2n)/f(n)$ of certain other summatory arithmetic functions \cite[p.\,180]{893}.
Asymptotics of sums of arithmetic functions over linearly recurrent sequences, like $2^k-1$, have been studied in larger generality \cite{ES96,LS07,Shp90}.
However, the current state-of-the-art still cannot handle sums of the classical arithmetic functions $\tau,\sigma,\phi$ (see Subsection \ref{subsec:notation} on the notation), but rather their ``tamed'' versions $\sigma(m)/m$ and $\phi(m)/m$.

Let us continue with a few words on the heuristics for $f(2n)/f(n)$, initially at a superficial level. 
Recall a trivial bound
\begin{equation}\label{eq:tauheur}
\frac{1}{n} \sum_{1\le k\le n} \tau(k) = \frac{1}{n} \sum_{1\le \ell\le n} \Big\lfloor \frac{n}{\ell} \Big\rfloor 
= \log n+O(1), 
\end{equation}
which had been improved already by Dirichlet; see Theorem 320 in Section 18.2 of the book \cite{HW}.
Because of the above asymptotics we can say that the function $\tau(n)$ behaves ``on the average'' as $\log n$. If $2^k-1$ behaved with respect to the number of divisors like a random integer of its size, then one could expect that $\tau(2^k-1)$ behaves ``on the average'' like $\log(2^k-1)$. 
Then
\begin{eqnarray*}
\sum_{1\le k\le n} \log(2^k-1) & = & \sum_{1\le k\le n} k\log 2+\sum_{1\le k\le n} \log\Big(1-\frac{1}{2^k}\Big)\\
& = & \frac{n(n+1)\log 2}{2}+O\Big(\sum_{1\le k\le n} \frac{1}{2^k}\Big)\\
& = & \frac{n^2\log 2}{2}+O(n)
\end{eqnarray*}
could lead us to further suspect that
$f(2n)/f(n)$ has a finite limit, which in fact should be equal to
$$
\lim_{n\to\infty} \frac{(1/2)(2n)^2\log2+O(n)}{(1/2)n^2\log2+O(n)} = 4.
$$

This informal reasoning makes plausible the conjecture that the ratio $f(2n)/f(n)$ stabilizes, but it is, at the same time, quite misleading. Namely, Theorem \ref{thm:limsup} below will imply that $\lim_{n\to\infty}f(2n)/f(n)$, if it even exists, cannot be a real number.
Numerical data presented in Section \ref{sec:numerics} will reveal that the above heuristics fails due to the fact that the summatory function $f$ of the quantities $\tau(2^k-1)$ is actually much ``less stable'' than predicted by \eqref{eq:tauheur}.
Nevertheless, the long term behaviour of $f(2n)/f(n)$ seems to be interesting and worth studying, both rigorously and experimentally, but we will have to employ better heuristics for the latter.


\subsection{Statements of the results}
Our first result simply claims that the doubling ratios $f(2n)/f(n)$ do not form a bounded sequence.

\begin{theorem}
\label{thm:limsup}
We have
$$
\limsup_{n\to\infty} \frac{f(2n)}{f(n)}=\infty,
$$
i.e., the sequence $(f(2n)/f(n))_{n\geq 1}$ is unbounded.
\end{theorem}

We will establish Theorem \ref{thm:limsup} by proving something stronger, but unrelated to the Mersenne numbers.
Namely, consider a modified summatory function defined as
$$
f'(n):=\sum_{1\le k\le n} 2^{\tau(k)}
$$
and rather ask whether $f'(2n)/f'(n)$ has a (possibly infinite) limit as $n\to\infty$.
Relationship between $f$ and $f'$ will be explained in Section \ref{sec:pfuncond}.
In particular, Inequality \eqref{eq:compare} below will show that $f'$ is, in fact, smaller than $f$, up to a multiplicative constant.
However, its doubling ratios turn out to be better behaved.

\begin{proposition}
\label{prop}
We have
$$
\lim_{n\to\infty} \frac{f'(2n)}{f'(n)}=\infty.
$$
\end{proposition}

The proof of Proposition \ref{prop} will be presented in Section \ref{sec:pfuncond} and it will use properties of highly--composite numbers, which were studied by Ramanujan \cite{Rama}, and also by Erd\H{o}s \cite{Erd1} and Nicolas \cite{Nic}.
Recall that a positive integer $N$ is \emph{highly--composite} if 
\[ \tau(N)>\tau(m) \] 
for all $1\le m<N$.
Theorem \ref{thm:limsup} will be easily derived from Proposition \ref{prop}.

\smallskip
The remaining problem is to decide whether $f(2n)/f(n)$ diverges to infinity, or it does not have any kind of limit.
We will also establish conjectural positive results in this direction, in the form of Theorem \ref{thm:cond} below, but only conditionally on,
\begin{itemize}
\item either a certain conjecture about the Mersenne numbers $2^N-1$ with more divisors than their predecessors (Conjecture \ref{conj1} below),
\item or a conjectured bound on the number of prime divisors of the $d$th cyclotomic polynomial evaluated at the number $2$ (Conjecture \ref{conj2} below).
\end{itemize}

To rigorously state the first conjecture we introduce a concept similar to highly--composite numbers, in which the attention is restricted to Mersenne numbers only.
Let us say that $N$ is an \emph{index of a highly--composite Mersenne number} if 
\[ \tau(2^N-1)>\tau(2^m-1) \] 
holds for all $1\le m<N$. They are motivated by the usefulness of highly--composite numbers in the proof of Proposition \ref{prop}. 
Note that the above definition does not mean that $2^N-1$ itself is a highly--composite number; it just means that $2^N-1$ has more divisors than any smaller Mersenne number $2^m-1$. Here is a reasonable conjecture about those numbers.

\begin{conjecture}
\label{conj1}
If $N$ is restricted to indices of highly--composite Mersenne numbers, then
\begin{equation}\label{eq:conjMratio}
\lim_{N\to\infty} \frac{\tau(2^N+1)}{N} = \infty.
\end{equation}
\end{conjecture}

An intuitive justification might be the following.
In Section \ref{sec:numerics} we will see that the number $\tau(2^N-1)$ is typically huge when $N$ is an index of a highly--composite Mersenne number. 
We expect this to be caused largely due to the fact that the polynomial $X^N-1$ splits into a large number of irreducible factors and, in turn, also that $N$ has many prime divisors.
In this situation, $X^N+1$ also splits into many factors and, consequently, $\tau(2^N+1)$ is necessarily at least somewhat large as well.
Some numerical evidence to support Conjecture \ref{conj1} will be presented in Section \ref{sec:numerics}.

To formulate the second conjecture recall that the $d$th \emph{cyclotomic polynomial} is denoted $\Phi_d$ and defined as
$$
\Phi_d(X):=\prod_{\substack{1\le h\le d\\ (h,d)=1}} (X-e^{2\pi i h/d}) \in\mathbb{Z}[X].
$$
We write $\omega(m)$ for the number of distinct prime factors of a positive integer $m$.

\begin{conjecture}
\label{conj2}
The inequality
\begin{equation}\label{eq:conjineq}
\omega(\Phi_d(2))\le 10\log d
\end{equation}
holds for all positive integers $d\ge 2$, with at most finitely many exceptions.
\end{conjecture}

A natural way of arriving at Conjecture \ref{conj2} and its heuristic confirmation will be presented in Subsection \ref{sec:heuronConj2}.

With either of the two assumptions, the limit \eqref{eq:conjMratio} or Inequality \eqref{eq:conjineq}, we can prove the claim on the divergence of the doubling ratios of $f$.

\begin{theorem}
\label{thm:cond}
If either of Conjectures \ref{conj1} and \ref{conj2} holds, then
\begin{equation}\label{eq:limisinfty}
\lim_{n\to\infty} \frac{f(2n)}{f(n)}=\infty.
\end{equation}
\end{theorem}

The proof of Theorem \ref{thm:cond} will span over the most of Section \ref{sec:pfcond}.

\smallskip
Finally, a large part of the present paper, namely the whole Section \ref{sec:numerics}, is dedicated to extensive computation and experimentation that support both the divergence claim $f(2n)/f(n)\to\infty$ and Conjectures \ref{conj1} and \ref{conj2}.

In Subsection \ref{sec:numexact} we evaluate, table, and plot the involved arithmetic quantities for as many positive integers $n$ as we can. In the process we use the readily available tables \cite{NoDivPlus,NoDivCyc,NoDivMinus} of the numbers $\tau(2^n-1)$ and $\omega(\Phi_n(2))$ for $n\leq 1206$, and the numbers $\tau(2^n+1)$ for $n\leq 1128$. This is about as far as the exact numerical evaluation can take us.
Namely, the \emph{Cunningham Project} \cite{CuPr}, initiated about a hundred years ago, attempts to factor numbers of the form $b^n\pm 1$ for $b\leq 12$ and as many values of $n$ as possible. Its typical approach is the Elliptic Curve Method \cite[Sec.~15.10]{Duj} and any further progress is expected to be rather slow.
However, even with this limited amount of data we illustrate the correctness of the hypothetical divergence claim and the above two conjectures.

Wanting to progress even further with the experimental support for $f(2n)/f(n)\to\infty$, we employ a mixture of an exact setting and a heuristic one to predict the numbers of divisors of large Mersenne numbers. More precisely, in Section \ref{sec:numprob} we partially factor $2^n-1$ into primes that are less than a given large threshold $A$ and then apply classical Gillies' approximation \cite{Gil64} to the number of its prime divisors greater than $A$.
We first test that our approximations to $\tau(2^n-1)$ are of the same order of magnitude as the exact values for $1000<n\leq1100$. Then we extend the previously obtained graphs of the doubling ratios $f(2n)/f(n)$ from $n\leq603$ all the way to $n\leq1000$, with the same conclusion as before.


\subsection{Notation}
\label{subsec:notation}
Throughout the paper we write $\tau(n)$ for the number of divisors of $n$, while $\omega(n)$ (resp.\@ $\Omega(n)$) is the number of prime factors of $n$ without (resp.\@ with) counting multiplicities.
The greatest common divisor of positive integers $m$ and $n$ is written simply as $(m,n)$, while $\phi(n)$ stands for the number of integers $m$ such that $1\leq m\leq n$ and $(m,n)=1$ (i.e., for \emph{Euler's totient function}).
Next, $\lfloor x\rfloor$ represents the largest integer not exceeding a real number $x$ (i.e., the \emph{floor} of $x$).
The logarithm with base $b$ is written as $\log_b$, while $\log$ means the natural logarithm, namely $\log_e$.
Notation $|S|$ is used for the number of elements of a finite set $S$.
If $g$ and $h$ are two real-valued functions defined on the set of positive integers, then we write $g(n) \ll h(n)$ and $h(n) \gg g(n)$ whenever $\limsup_{n\to\infty}|g(n)/h(n)|<\infty$.
Likewise, $g(n)=O(h(n))$ is yet another way of writing $g(n)\ll h(n)$, while $g(n)=o(h(n))$ simply means that $\lim_{n\to\infty}g(n)/h(n)=0$.


\section{Proofs of unconditional results}
\label{sec:pfuncond}
One can naturally arrive at the quantity $2^{\tau(k)}$, when starting from $\tau(2^k-1)$. 
Namely, recall that Bang \cite{Bang} (and also Zsigmondy \cite{Zs}) proved that $2^m-1$ has a primitive prime divisor (a prime factor $p$ that does not divide $2^{\ell}-1$ for any positive integer $\ell<m$) for all positive integers $m$ except for $m=1,6$. Since $2^d-1\mid 2^k-1$ for all divisors $d$ of $k$, it follows that $2^k-1$ has at least as many prime factors as
the number of divisors of $k$ different than $1$ or $6$. Thus, $\omega(2^k-1)\ge \tau(k)-2$, which leads to
\begin{equation}
\label{eq:compare0}
\tau(2^k-1)\ge 2^{\omega(2^k-1)}\ge \frac{1}{4} 2^{\tau(k)}
\end{equation}
and then to
\begin{equation}
\label{eq:compare}
f(n) \ge \frac{1}{4} f'(n).
\end{equation}

Let us now turn to the result concerned with the function $f'$.

\begin{proof}[Proof of Proposition \ref{prop}]
It suffices to show that
$$
\frac{\sum_{n<k\le 2n} 2^{\tau(k)}}{\sum_{1\le k\le n} 2^{\tau(k)}}
$$
tends to infinity. Let $N$ be the largest highly--composite number not exceeding $n$. Clearly, $N>n/2$ because $\tau(2m)>\tau(m)$ holds for any positive integer $m$.
Then $2N\in (n,2n]$, so
\begin{equation}\label{eq:thefirst}
\frac{\sum_{n<k\le 2n} 2^{\tau(k)}}{\sum_{1\le k\le n} 2^{\tau(k)}} \ge \frac{2^{\tau(2N)}}{n2^{\tau(N)}} > \frac{2^{\tau(2N)}}{2N2^{\tau(N)}} =2^{\tau(2N)-\tau(N)+O(\log N)}.
\end{equation}
It is known that
\begin{equation}\label{eq:thenodiv}
\tau(N)=2^{(1+o(1))\log N/\log\log N}\qquad {\text{\rm as }} N\to\infty.
\end{equation}
Namely, the upper bound holds for completely general $N$ due to a classical result by Wigert (see Theorem 317 in \cite{HW}), while the lower bound is specific to highly--composite numbers, among other alike sequences (and it easily follows from Ramanujan's considerations in \S5 of \cite{Rama}).\footnote{If $p_k$ stands for the $k$th prime and $N$ is a highly--composite number such that $2\cdot 3\cdots p_t \leq N<2\cdot 3\cdots p_t p_{t+1}$ for some $t\geq1$, then $\tau(N)\geq \tau(2\cdot 3\cdots p_t) = 2^t = 2^{\pi(p_t)}$, where $\pi(x):=|\{p\leq x:p\text{ prime}\}|$ denotes the prime-counting function. On the other hand, $\log N = (1+o(1)) \vartheta(p_t)$, where $\vartheta(x):=\sum_{p\leq x, p\,\text{prime}}\log p$ is the Chebyshev function, so the lower bound in \eqref{eq:thenodiv} follows from basic asymptotic properties of $\pi$ and $\vartheta$ or from Formula (25) in \S5 of \cite{Rama}.}
Writing $N$ as in \S8 of \cite{Rama},
$$
N=2^{e_1} 3^{e_2}\cdots p_t^{e_t},\qquad {\text{\rm with }} e_1\ge e_2\ge \cdots\ge e_t,
$$
where $p_k$ is the $k$th prime, we have
$$
\tau(N)=(e_1+1)(e_2+1)\cdots(e_t+1)\qquad {\text{\rm and}}\qquad \tau(2N)=(e_1+2)(e_2+1)\cdots(e_t+1).
$$
Thus,
$$
\tau(2N)-\tau(N)=(e_2+1)\cdots(e_t+1)=\frac{\tau(N)}{e_1+1}\gg \frac{\tau(N)}{\log N}.
$$
Hence, using \eqref{eq:thenodiv} we conclude that the exponent in \eqref{eq:thefirst} satisfies
\begin{eqnarray*}
\tau(2N)-\tau(N)+O(\log N) & \gg &  \frac{\tau(N)}{\log N}+O(\log N)\\
& = & \frac{2^{(1+o(1))\log N/\log\log N}}{\log N}+O(\log N)\\
& = & 2^{(1+o(1))\log N/\log\log N}\\
& = & 2^{(1+o(1))\log n/\log\log n},
\end{eqnarray*}
so it tends to infinity as $n\to\infty$, which is what we wanted.
\end{proof}

Our main unconditional result, namely Theorem \ref{thm:limsup}, is now an immediate consequence of Proposition \ref{prop}.

\begin{proof}[Proof of Theorem \ref{thm:limsup}]
If the sequence of ratios $f(2n)/f(n)$ was bounded from above by some $C<\infty$, then we would have $f(2^m)\leq C^{m}$ for every integer $m\geq1$.
On the other hand, by \eqref{eq:compare} and Proposition \ref{prop},
$$
f(2^m) \geq \frac{1}{4} f'(2^m) = \frac{1}{2} \prod_{\ell=0}^{m-1} \frac{f'(2^{\ell+1})}{f'(2^\ell)}
$$
grows faster than $C^m$ as $m\to\infty$ for any finite constant $C$.
\end{proof}


\section{Proofs of conditional results}
\label{sec:pfcond}

In the next stage one may wonder about the true nature of the function $f$. 
\begin{itemize}
\item Is it more likely that the numbers $f(n)$ occasionally grow like a polynomial in $n$, which would be the case if $\tau(2^k-1)$ behaved on the average like $\log(2^k-1)$ as suggested by the introductory computation \eqref{eq:tauheur}, or 
\item is it more likely that $f(2n)/f(n)$ tends to infinity, as it was the case for the function $f'$ obtained by replacing $\tau(2^k-1)$ with $2^{\tau(k)}$?
\end{itemize}

\subsection{Reduction to highly--composite Mersenne numbers}
The proof of the first conditional result is similar to the proof of Proposition \ref{prop}.

\begin{proof}[Proof of Theorem \ref{thm:cond} assuming Conjecture \ref{conj1}]
It suffices to consider
$$
\frac{\sum_{n<k\le 2n} \tau(2^k-1)}{\sum_{1\le k\le n} \tau(2^k-1)}.
$$
Let $N$ be the largest index of a highly--composite Mersenne number that satisfies $N\leq n$. Clearly, $N>n/2$, because $2^{2m}-1=(2^m-1)(2^m+1)$ and so $\tau(2^{2m}-1)>\tau(2^m-1)$. We then have $2N\in (n,2n]$ and thus,
$$
\frac{\sum_{n<k\le 2n} \tau(2^k-1)}{\sum_{1\le k\le n} \tau(2^k-1)}\ge \frac{\tau(2^{2N}-1)}{n\tau(2^N-1)}\ge \frac{\tau((2^N-1)(2^N+1))}{2N\tau(2^N-1)}=\frac{\tau(2^N+1)}{2N}.
$$
If we assume Conjecture \ref{conj1}, we see that the last expression diverges to $\infty$ as $n\to\infty$ (and thus also $N\to\infty$).
\end{proof}

\subsection{Heuristic justification of Conjecture \ref{conj2}}
\label{sec:heuronConj2}
Our next goal is to arrive naturally at Conjecture \ref{conj2}.
Write
$$
2^n-1=\prod_{i=1}^t q_i^{e_i}
$$
where $q_1<\cdots<q_t$ are distinct primes and $e_1,\ldots,e_t$ are positive integer exponents.
Clearly,
$$
e_i\leq \frac{\log(2^n-1)}{\log q_i} < \frac{\log 2^n}{\log 2} = n.
$$
Better bounds are known (see, for example, Theorem 1.2 in Stewart \cite{Ste}), but we will not need them. We thus have that
\begin{equation}
\label{eq:k}
\tau(2^n-1)=\prod_{i=1}^t (e_i+1)\le n^t = 2^{t\log n/\log 2} < 2^{2\omega(2^n-1)\log n}.
\end{equation}
The bottleneck of our approach consists of getting a good upper bound on $\omega(2^n-1)$. We factor $X^n-1$ into cyclotomic polynomials as
$$
\prod_{d\mid n} \Phi_d(X)
$$
and thus obtain
$$
2^n-1=\prod_{d\mid n} \Phi_d(2).
$$
In particular,
\begin{equation}\label{eq:omega}
\omega(2^n-1) \leq \sum_{d\mid n} \omega(\Phi_d(2)).
\end{equation}
We do not expect much overcounting on the right-hand side of \eqref{eq:omega}, since it is known that all prime factors of $\Phi_d(2)$ are primes $p$ such that the multiplicative order of $2$ modulo $p$ is exactly $d$, which are mutually disjoint for different $d$, except that $\Phi_d(2)$ might also be divisible by the largest prime factor of $d$. 

It remains to make an educated guess concerning $\omega(\Phi_d(2))$. Trivially, from the definition of $\Phi_d$,
\begin{equation}\label{eq:trivialforPhi}
\Phi_d(2)\le 3^{\phi(d)}.
\end{equation}
Using the well-known bound\footnote{We have
$$ \frac{\log n}{\log\log n} \geq \frac{\log(2\cdot 3\cdots p_{\omega(n)})}{\log\log(2\cdot 3\cdots p_{\omega(n)})}
= \frac{\vartheta(p_{\omega(n)})}{\log\vartheta(p_{\omega(n)})} \gg \pi(p_{\omega(n)}) = \omega(n), $$
where we again used an easy bound (25) from \S5 of \cite{Rama}.}
\[ \omega(n) \ll \frac{\log n}{\log\log n}, \]
we arrive at
$$
\omega(\Phi_d(2))\ll \frac{\log 3^{\phi(d)}}{\log\log 3^{\phi(d)}}\ll \frac{\phi(d)}{\log \phi(d)},
$$
but the last estimate is too large to be useful in combination with \eqref{eq:omega}.
Thus, we are led to impose a stronger assumption like a logarithmic bound from Conjecture \ref{conj2}.

\smallskip
Now we offer some heuristic support for Conjecture \ref{conj2}, while numerical experimentation is postponed to Section \ref{sec:numerics}. Exercise 0.4 on page 12 in the book by Hall and Tenenbaum \cite{div} shows that, uniformly in $K\ge 1$ and $x\ge 2$,
$$
|\{n\le x: \Omega(n)=K\}|\ll \frac{xK \log x}{2^K}.
$$
In particular,
$$
|\{n\le x: \Omega(n)> K\}|\ll x\log x\sum_{k> K} \frac{k}{2^k}\ll \frac{Kx\log x}{2^K}.
$$
We can interpret this by saying that, for a ``random'' positive integer $n$, the probability of having more than $K$ prime factors (even counting multiplicities) is
$$
\ll \frac{K\log n}{2^K}.
$$
We apply this heuristic with $n=\Phi_d(2)$ and $K=10\log d$, and also use \eqref{eq:trivialforPhi}, getting that the probability of $\omega(\Phi_d(2))> 10\log d$ is
$$
\ll \frac{(10\log d) \log(\Phi_d(2))}{2^{10\log d}}\ll \frac{\phi(d)\log d}{d^{10\log 2}}\leq \frac{\log d}{d^{10\log 2-1}}\ll \frac{1}{d^5}.
$$
Since the series
$$
\sum_{d\ge 2} \frac{1}{d^5}
$$
converges, it may be indeed the case, in analogy with the probabilistic Borel--Cantelli lemma, that there are only finitely many positive integers $d$ that violate \eqref{eq:conjineq}. Therefore, Conjecture \ref{conj2} is likely to hold.

\subsection{Reduction to cyclotomic polynomials}
Let us now prove the other conditional result.

\begin{proof}[Proof of Theorem \ref{thm:cond} assuming Conjecture \ref{conj2}]
Suppose that $d_0\geq2$ is an integer such that \eqref{eq:conjineq} holds for every integer $d\geq d_0$.
It is sufficient to verify conditionally Conjecture \ref{conj1} and the claim \eqref{eq:limisinfty} will then follow from the previous conditional result.
Thus, suppose that $N$ is an index of a highly--composite Mersenne number.

We need to say something about $\tau(2^N+1)$, so let us first take a closer look at $\tau(2^N-1)$. By inequality \eqref{eq:k} we have
$$
\tau(2^N-1)\le 2^{2\omega(2^N-1)\log N},
$$
while, by estimates \eqref{eq:conjineq} and \eqref{eq:omega},
\begin{align*}
\omega(2^N-1) & \le \sum_{d< d_0} \omega(\Phi_d(2))+\sum_{\substack{d\ge d_0\\ d\mid N}} 10\log d \\
& \leq O(1) + 10\tau(N)\log N \ll \tau(N)\log N .
\end{align*}
Thus, under Conjecture \ref{conj2} we know
\begin{equation}
\label{eq:100}
\tau(2^N-1)\le 2^{O(\tau(N)(\log N)^2)}.
\end{equation}
On the other hand, by the fact that $N$ is an index of a highly--composite Mersenne number and relation \eqref{eq:compare0}, we have that
$$
\tau(2^N-1)\ge \max_{1\leq k\leq N} \tau(2^k-1)\ge \max_{1\leq k\leq N} 2^{\tau(k)-2}.
$$
Choose $K$ to be the largest highly--composite number not exceeding $N$. Then $K>N/2$, and by \eqref{eq:thenodiv} we get
\begin{equation}
\label{eq:100b}
\log_2 \tau(2^N-1)\ge \tau(K)-2=2^{(1+o(1))\log K/\log\log K}=2^{(1+o(1))\log N/\log\log N}.
\end{equation}
Equations \eqref{eq:100} and \eqref{eq:100b} now show that
$$
\tau(N)(\log N)^2\gg 2^{(1+o(1))\log N/\log\log N},
$$
so
$$
\tau(N)\gg 2^{(1+o(1))\log N/\log\log N}.
$$
Write $N=2^a M$, where $M$ is odd and $a\geq0$ is an integer. Then we have
$$
\tau(N)=(a+1)\tau(M)\ll \tau(M)\log N,
$$
which implies
$$
\tau(M)\ge 2^{(1+o(1))\log N/\log\log N}.
$$
The number $N$ has $\tau(M)$ divisors $d$ such that $N/d$ is odd and $2^d+1\mid 2^N+1$ for all such divisors $d$. By Bang's theorem \cite{Bang} again,
$$
\tau(2^N+1)\ge 2^{\tau(M)-2}.
$$
Hence, we get that
$$
\frac{\tau(2^N+1)}{N}\ge \frac{2^{\tau(M)-2}}{N}=2^{\tau(M)+O(\log N)}
\geq 2^{2^{(1+o(1))\log N/\log\log N}},
$$
which diverges to infinity as $N\to\infty$.
\end{proof}


\section{Experimental evidence}
\label{sec:numerics}

\subsection{Exact numerical computation}
\label{sec:numexact}

\begin{figure}
\includegraphics[width=0.55\linewidth]{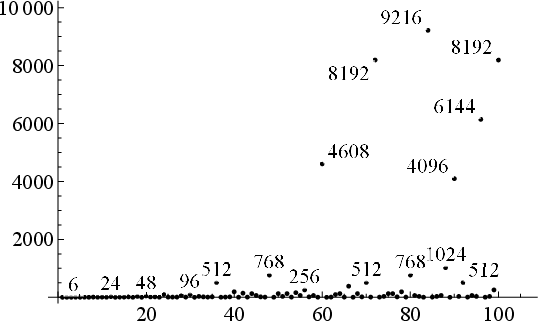}
\caption{Values of $\tau(2^n-1)$ for $1\leq n\leq 100$.}
\label{fig:ratio1}
\end{figure}

The number of divisors of Mersenne numbers behaves quite ``erratically'' and exhibits occasional huge jumps; see Figure \ref{fig:ratio1}.
In order to perform more substantial experimentation, one can compute $\tau(2^n-1)$ for $1\leq n\leq 250$ on a home computer using the Wolfram \emph{Mathematica} \cite{WM} command \verb+DivisorSigma[0,2^n-1]+, but for larger values of $n$ it is necessary to use faster tools for integer factorization, such as Alpern's integer factorization calculator \emph{Alpertron} \cite{Alp}.
A long-running project dedicated to factoring numbers of the form $b^n\pm 1$ for small $b$ and large $n$ is the \emph{Cunningham Project} \cite{CuPr}, while a ready-to-use list of $\tau(2^n-1)$ for $n\leq 1206$ was produced by Eldar \cite{NoDivMinus}, as a part of the entry for the Encyclopedia \emph{OEIS} \cite{OEIS} sequence A046801.

\begin{figure}
\includegraphics[width=0.47\linewidth]{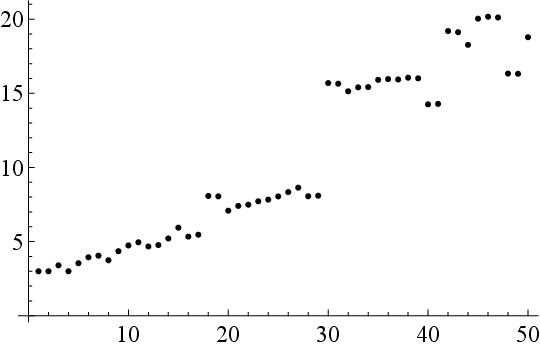}
\includegraphics[width=0.47\linewidth]{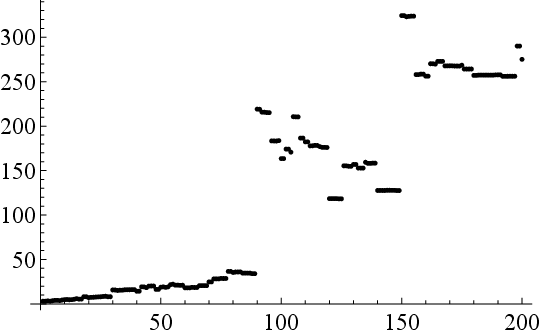}
\includegraphics[width=0.47\linewidth]{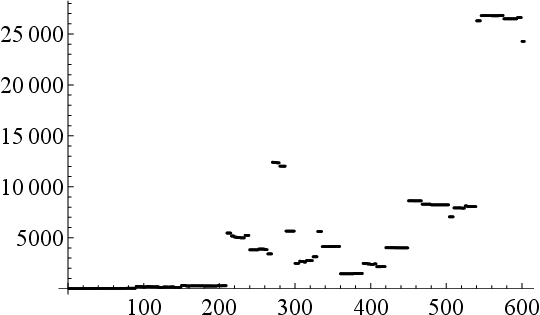}
\caption{Numerical values of the ratios $f(2n)/f(n)$ for $1\leq n\leq 50$, $1\leq n\leq 200$, and $1\leq n\leq 603$, respectively.}
\label{fig:ratio234}
\end{figure}

Using \cite{NoDivMinus} we calculated the exact values of $f(2n)/f(n)$ for $1\leq n\leq 603$ and illustrated them in Figure \ref{fig:ratio234}. It seems that these ratios exhibit occasional large jumps responsible for $f(2n)/f(n)\to\infty$. What we appear to be observing is that, as soon as $2n\geq k$ for a Mersenne number $2^k-1$ with a record-breaking number of divisors, the quotient $f(2n)/f(n)$ also becomes large, and it does not have time to decrease too much before the next huge jump of $\tau(2^k-1)$ occurs.

\begin{table}
\begin{center}
\begin{tabular}[t]{|r|r|c|}
\hline
$N$ & $\tau(2^N-1)$ & $\tau(2^N+1)/N$ \\
\hline\hline
$1$ & $1$ & $2$ \\
\hline
$2$ & $2$ & $1$ \\
\hline
$4$ & $4$ & $0.5$ \\
\hline
$6$ & $6$ & $0.6667$ \\
\hline
$8$ & $8$ & $0.25$ \\
\hline
$12$ & $24$ & $0.3333$ \\
\hline
$18$ & $32$ & $0.8889$ \\
\hline
$20$ & $48$ & $0.2$ \\
\hline
$24$ & $96$ & $0.3333$ \\
\hline
$36$ & $512$ & $0.4444$ \\
\hline
$48$ & $768$ & $0.1667$ \\
\hline
$60$ & $4\,608$ & $0.2667$ \\
\hline
$72$ & $8\,192$ & $0.4444$ \\
\hline
$84$ & $9\,216$ & $0.3810$ \\
\hline
$108$ & $10\,240$ & $0.5926$ \\
\hline
\end{tabular}
\quad
\begin{tabular}[t]{|r|r|c|}
\hline
$N$ & $\tau(2^N-1)$ & $\tau(2^N+1)/N$ \\
\hline\hline
$120$ & $73\,728$ & $0.5333$ \\
\hline
$144$ & $262\,144$ & $0.8889$ \\
\hline
$168$ & $294\,912$ & $1.5238$ \\
\hline
$180$ & $6\,291\,456$ & $1.4222$ \\
\hline
$288$ & $33\,554\,432$ & $0.4444$ \\
\hline
$300$ & $100\,663\,296$ & $1.7067$ \\
\hline
$360$ & $1\,610\,612\,736$ & $2.8444$ \\
\hline
$420$ & $57\,982\,058\,496$ & $9.7524$ \\
\hline
$540$ & $257\,698\,037\,760$ & $60.6815$ \\
\hline
$660$ & $463\,856\,467\,968$ & $397.1879$ \\
\hline
$720$ & $1\,649\,267\,441\,664$ & $45.5111$ \\
\hline
$780$ & $1\,855\,425\,871\,872$ & $42.0103$ \\
\hline
$840$ & $237\,494\,511\,599\,616$ & $624.1524$ \\
\hline
$900$ & $281\,474\,976\,710\,656$ & $36.4089$ \\
\hline
$1080$ & $8\,444\,249\,301\,319\,680$ & $242.7259$ \\
\hline
\end{tabular}
\end{center}
\caption{The first $30$ indices of highly--composite Mersenne numbers.}
\label{tab:hcMersenne}
\end{table}

We turn to Conjecture \ref{conj1}. In Table \ref{tab:hcMersenne} we made a list of all positive integers $N\le 1206$ that are indices of highly--composite Mersenne numbers. There are $30$ such values of $N$ and they are also listed as sequence A177710 in the Encyclopedia \emph{OEIS}.
In the last column we also computed the corresponding ratios $\tau(2^N+1)/N$ using Alekseyev's list \cite{NoDivPlus} of the numbers $\tau(2^n+1)$ for $n\leq 1128$, which is a part of the \emph{OEIS} entry A046798 for this sequence.
The table shows the trend $\tau(2^N+1)/N\to\infty$, even though this is somewhat open to interpretation, as we quickly run out of available data.

\begin{table}
\begin{center}
\begin{tabular}[t]{|r|r|c|}
\hline
$d$ & $\omega(\Phi_d(2))$ & $\omega(\Phi_d(2))/\log d$ \\
\hline\hline
$1$ &  $0$ &  \\
\hline
$2$ &  $1$ &  $1.4427$ \\
\hline
$3$ &  $1$ &  $0.9102$ \\
\hline
$4$ &  $1$ &  $0.7213$ \\
\hline
$5$ &  $1$ &  $0.6213$ \\
\hline
$6$ &  $1$ &  $0.5581$ \\
\hline
$7$ &  $1$ &  $0.5139$ \\
\hline
$8$ &  $1$ &  $0.4809$ \\
\hline
$9$ &  $1$ &  $0.4551$ \\
\hline
$10$ &  $1$ &  $0.4343$ \\
\hline
$11$ &  $2$ &  $0.8341$ \\
\hline
$12$ &  $1$ &  $0.4024$ \\
\hline
$13$ &  $1$ &  $0.3899$ \\
\hline
$14$ &  $1$ &  $0.3789$ \\
\hline
$15$ &  $1$ &  $0.3693$ \\
\hline
$16$ &  $1$ &  $0.3607$ \\
\hline
$17$ &  $1$ &  $0.3530$ \\
\hline
$18$ &  $2$ &  $0.6920$ \\
\hline
$19$ &  $1$ &  $0.3396$ \\
\hline
$20$ &  $2$ &  $0.6676$ \\
\hline
\end{tabular}
\quad
\begin{tabular}[t]{|r|r|c|}
\hline
$d$ & $\omega(\Phi_d(2))$ & $\omega(\Phi_d(2))/\log d$ \\
\hline\hline
$21$ &  $2$ &  $0.6569$ \\
\hline
$22$ &  $1$ &  $0.3235$ \\
\hline
$23$ &  $2$ &  $0.6379$ \\
\hline
$24$ &  $1$ &  $0.3147$ \\
\hline
$25$ &  $2$ &  $0.6213$ \\
\hline
$26$ &  $1$ &  $0.3069$ \\
\hline
$27$ &  $1$ &  $0.3034$ \\
\hline
$28$ &  $2$ &  $0.6002$ \\
\hline
$29$ &  $3$ &  $0.8909$ \\
\hline
$30$ &  $1$ &  $0.2940$ \\
\hline
$31$ &  $1$ &  $0.2912$ \\
\hline
$32$ &  $1$ &  $0.2885$ \\
\hline
$33$ &  $1$ &  $0.2860$ \\
\hline
$34$ &  $1$ &  $0.2836$ \\
\hline
$35$ &  $2$ &  $0.5625$ \\
\hline
$36$ &  $2$ &  $0.5581$ \\
\hline
$37$ &  $2$ &  $0.5539$ \\
\hline
$38$ &  $1$ &  $0.2749$ \\
\hline
$39$ &  $2$ &  $0.5459$ \\
\hline
$40$ &  $1$ &  $0.2711$ \\
\hline
\end{tabular}
\end{center}
\caption{Numbers of prime divisors of $\Phi_d(2)$ for $d\leq 40$.}
\label{tab:cyclo}
\end{table}

Now we turn to Conjecture \ref{conj2}. Table \ref{tab:cyclo} presents $\omega(\Phi_d(2))$ and $\omega(\Phi_d(2))/\log d$ for $d\leq40$, but this amount of data is too small to conclude anything about an upper bound for the latter quotient.
\emph{Mathematica} commands \verb+PrimeNu[Cyclotomic[d,2]]+ can again easily compute numbers of prime factors of $\Phi_d(2)$ for $d$ up to about $250$. Luckily, the numbers $\Phi_d(2)$ and $\omega(\Phi_d(2))$ are respectively listed as sequences A019320 and A085021 in \emph{OEIS}.
For the latter one Alekseyev produced a table \cite{NoDivCyc} of values $\omega(\Phi_d(2))$ for $d\leq1206$, and this list is depicted in Figure \ref{fig:cyclo12}. It is now reasonable to suspect a logarithmic upper bound in $d$ on the numbers $\omega(\Phi_d(2))$.
Namely, direct verification shows
\[ \omega(\Phi_d(2)) \leq 1.51 \log d \quad\text{for every } 1\leq d\leq 1206, \]
providing some evidence for Conjecture \ref{conj2}.

\begin{figure}
\includegraphics[width=0.47\linewidth]{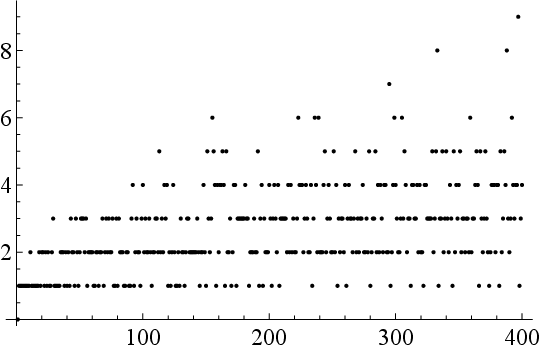}
\includegraphics[width=0.47\linewidth]{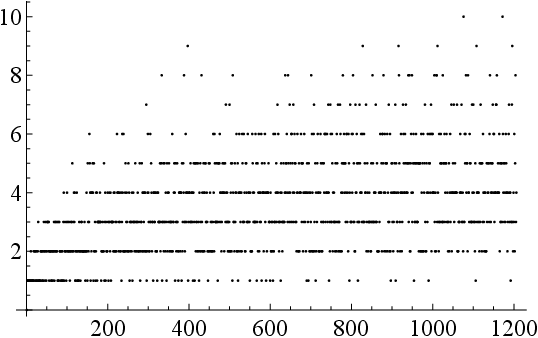}
\caption{Values of $\omega(\Phi_d(2))$ for $1\leq d\leq 400$ and $1\leq d\leq 1206$, respectively.}
\label{fig:cyclo12}
\end{figure}


\subsection{Partially heuristic model}
\label{sec:numprob}
Finally, in order to check the growth of $f(2n)/f(n)$ for even larger integers $n$, we combine the exact computation with an informal but reasonable heuristic.

Let us factor partially,
\[ 2^n - 1 = p_1^{e_1} p_2^{e_2} \cdots p_t^{e_t} M, \]
where $p_1,\ldots,p_t$ are all primes less than a given threshold $A$, while $M$ is not divisible by any of them. The threshold $A$ is allowed to depend on $n$ and it is chosen such that the partial factorization completes in under a minute on a home computer with standard performances.
Then we reason that a ``random'' integer in an interval $[t,t+\Delta t]$ is prime with probability about $1/\log t$, by the prime number theorem, as soon as $\Delta t$ is large but much smaller than $t$. The probability that such a number is both prime and divides a given integer $M\gg t$ should then be about $1/(t\log t)$ and, consequently, the number $M$ should have about
\[ \lambda = \int_{A}^{\sqrt{M}} \frac{\textup{d}t}{t\log t}  = \log\frac{\log\sqrt{M}}{\log A} \]
prime divisors in the interval $[A,\sqrt{M}]$.
We expect these prime divisors to typically be different, so a reasonable approximation to the number of divisors of Mersenne numbers is
\begin{equation}\label{eq:mersapprox}
\tau(2^n - 1) \approx (e_1+1)(e_2+1)\cdots(e_t+1) 2^\lambda.
\end{equation}
Note that already Gillies \cite{Gil64} used this reasoning to approximate the number of divisors of the numbers $2^p-1$ for prime $p$ and even conjectured that this number behaves like a Poisson-distributed random variable with mean $\lambda$. His probabilistic model was statistically tested, criticized and refined by subsequent authors \cite{Ehr67,Wag83}, but those subtleties do not seem to make a difference here, as soon as we are taking $A$ to be quite large.

In order to test how good approximation \eqref{eq:mersapprox} really is, we compare it with the true values of $\tau(2^n - 1)$ for $1001\leq n\leq 1100$. Base $10$ logarithms of the ratios 
\[ \frac{\text{approximate value}}{\text{true value}} \]
are depicted in Figure \ref{fig:testratio} and we can really see that the right-hand side is within one order of magnitude of the left-hand side, which is quite satisfactory.
Thus, we append the list of the exact values of $\tau(2^n-1)$ for $n\leq 1206$ with their approximations for $1207\leq n\leq 2000$.
Then we compute the (now approximate) values of the doubling ratios $f(2n)/f(n)$ for $n\leq 1000$ and graph them in Figure \ref{fig:approxratios}.
The left graph suggests that the trend of ``sudden massive jumps'' continues, even though we again see many drops while awaiting for the next huge jump. The right graph simply zooms in the ratios for the indices $901\leq n\leq 1000$. Even though these ratios, which are around $80000$, are much smaller than the new peak value of about $5$ million (and also smaller than a few previous peak values), they are still larger than the peak value of about $25000$ from Figure \ref{fig:ratio234}. All this provides some more evidence for $f(2n)/f(n)\to\infty$ and, at the same time, exposes the subtlety of this claim.

\begin{figure}
\includegraphics[width=0.48\linewidth]{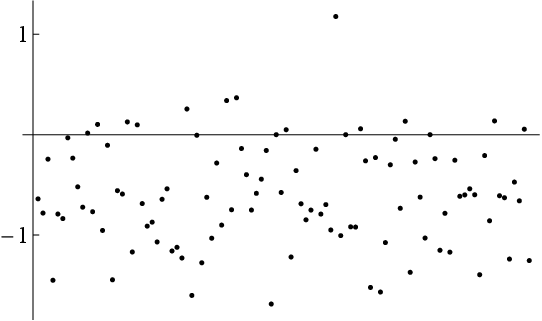}
\caption{Base $10$ logarithms of the ratios between the approximate and the true values of $\tau(2^n-1)$ for $1001\leq n\leq 1100$.}
\label{fig:testratio}
\end{figure}

\begin{figure}
\includegraphics[width=0.5\linewidth]{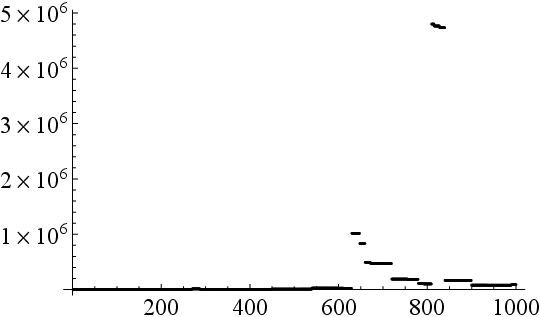}
\includegraphics[width=0.48\linewidth]{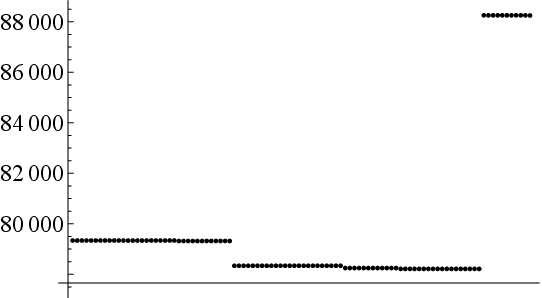}
\caption{Numerical values of the approximately evaluated ratios $f(2n)/f(n)$ for $1\leq n\leq 1000$ (left) and the last $100$ points zoomed (right).}
\label{fig:approxratios}
\end{figure}


\section*{Acknowledgments}
V. K. was supported by the Croatian Science Foundation under the project number HRZZ-IP-2022-10-5116 (\emph{FANAP}).
F. L. was supported in part by the ERC Synergy Project \emph{DynAMiCs}.
The authors are grateful to Andrej Dujella for informing them about the status of \emph{The Cunningham Project} \cite{CuPr} and to Michel Marcus for references to the very useful extensive data-lists \cite{NoDivPlus,NoDivCyc,NoDivMinus}.
They also want to thank Stijn Cambie for expressing a similar heuristic belief on the divergence of $f(2n)/f(n)$.
Finally, the authors are indebted to the referees, who suggested numerous improvements to the material and its exposition.




\end{document}